\documentclass[10pt]{article}

\usepackage[top=3cm,bottom=3cm,left=4cm,right=4cm]{geometry}
\usepackage{amsmath,amscd,amssymb,amsthm}
\usepackage[english]{babel}
\usepackage{mathtools}
\usepackage{enumerate}
\usepackage[numbers]{natbib}
\usepackage[hidelinks]{hyperref}

\title{The Density of Shifted and Affine Eisenstein Polynomials}
\author{Giacomo Micheli and Reto Schnyder}

\DeclarePairedDelimiter{\abs}{\lvert}{\rvert}
\DeclarePairedDelimiter{\ceil}{\lceil}{\rceil}

\DeclareMathOperator{\disc}{Disc}
\DeclareMathOperator{\GL}{GL}

\def\vF{\mathbb{F}}

\def\vZ{\mathbb{Z}}

\def\vR{\mathbb{R}}
\def\vQ{\mathbb{Q}}

\newtheorem{theorem}{Theorem}
\newtheorem{question}{Question}
\newtheorem{lemma}[theorem]{Lemma}

\newtheorem{proposition}[theorem]{Proposition}
\theoremstyle{definition}
\newtheorem{definition}[theorem]{Definition}

\newtheorem{remark}[theorem]{Remark}

\begin{document}

\maketitle

\begin{abstract}
In this paper we provide a complete answer to a question by Heyman and
Shparlinski concerning the natural density of polynomials which are irreducible
by Eisenstein's criterion after applying some shift.
The main tool we use is a local to global principle for density computations
over a free $\vZ$-module of finite rank.
\end{abstract}

\section{Introduction}
\label{sec:introduction}

Let $\vZ$ be the ring of rational integers.
The Eisenstein irreducibility criterion~\cite{bib:eisenstein,bib:schoenemann} is
a very convenient tool to establish that a polynomial in $\vZ[x]$ is
irreducible.
It is a well understood fact that the density of irreducible polynomials of
fixed degree $d$ among all the polynomials of degree $d$ is equal to one.
The question which naturally arises is the following:
\begin{question}\label{qu:normal_Eis}
	What is the density of polynomials which are irreducible by the Eisenstein
	criterion?
\end{question}
More informally, how likely is it that checking whether a random polynomial is
irreducible using only the Eisenstein irreducibility criterion leads to success?
In \citep{bib:PolyDub,bib:shparlinskiEisen} the authors deal with Eisenstein
polynomials of fixed degree with coefficients over $\vZ$.
They provide a complete answer to the above question in the case of monic
(See~\citep{bib:PolyDub}, \cite[Theorem~1]{bib:shparlinskiEisen}) and non-monic
(See~\cite[Theorem~2]{bib:shparlinskiEisen}) Eisenstein polynomials.

From now on we will specialize to the case of non-monic Eisenstein polynomials,
since the proofs and methods can be easily adapted from one case to the other.
In~\cite{bib:shparlinskiEisen}, the authors consider the set of polynomials of
degree at most $d$ having integer coefficients bounded in absolute value by $B$
(the \emph{height} of a polynomial) and give a sharp estimate for the number
$\rho(B)$ of polynomials which are irreducible by the Eisenstein criterion.
The \emph{natural density} of Eisenstein polynomials is then the limit of the
sequence $\rho(B)/(2B)^{d+1}$, which fully answers Question~\ref{qu:normal_Eis}.

As is well known, a polynomial $f(x)$ is irreducible if and only if $f(x+i)$ is
irreducible for all $i\in \vZ$.
Using this simple observation, one could check irreducibility by trying to use
the Eisenstein criterion for many $i$.
How likely is it that this procedure works?
More formally,
\begin{question}\label{qu:shiftedquestion}
	What is the natural density of polynomials $f(x)$ for which $f(x+i)$ is
	irreducible by the Eisenstein criterion for some integer shift $i$?
\end{question}

In~\citep{bib:heyman2014shifted}, Heyman and Shparlinski address this question,
giving a lower bound on this density.
Nevertheless, the question regarding the exact density remained open.
In this paper, we provide a complete solution to
Question~\ref{qu:shiftedquestion} using a local to global principle for
densities~\citep[Lemma~20]{bib:poonenAnn}.
Using similar methods, we also provide a solution to the question appearing in
\cite[Section~7]{bib:heyman2014shifted} about \emph{affine} Eisenstein
polynomials.

Our proofs are also supported by Monte Carlo experiments which we provide in
Section~\ref{sec:simulations}.

\section{Notation}
\label{sec:notation}
I this section we fix the notation that will be used throughout the paper.
\begin{definition}
	Let $R$ be an integral domain and $R[x]$ be the ring of polynomials with
	coefficients in $R$.
	We say that $f(x)=\sum^n_{i=0} \alpha_i x^i\in R[x]$ of degree $n$ is
	\emph{Eisenstein with respect to a prime ideal $p$} or \emph{$p$-Eisenstein}
	if
	\begin{itemize}
		\item $\alpha_n\notin p$.
		\item $\alpha_i\in p$ for all $i\in \{0,\dots,n-1\}$.
		\item $\alpha_0\notin p^2$.
	\end{itemize}
	We say that $f(x)$ is \emph{Eisenstein} if it is Eisenstein with respect to
	some prime ideal $p$.
\end{definition}

In this paper, we will only consider the ring of integers $R = \vZ$ and
the rings of $p$-adic integers $R = \vZ_p$.

\begin{definition}
	For any subset $A\subseteq \vZ^d$, we define
	\begin{align*}
		\overline{\rho}(A) &:=
			\limsup_{B \to \infty}\frac{\abs[\big]{A \cap [-B,B[^d}}{(2B)^d}, \\
		\underline{\rho}(A) &:=
			\liminf_{B \to \infty}\frac{\abs[\big]{A \cap [-B,B[^d}}{(2B)^d}. \\
		\intertext{If these coincide, we denote their value by $\rho(A)$ and
			call it the \emph{natural density} of $A$:}
		\rho(A) &:= \smashoperator[r]{\lim_{B \to \infty}}
			\frac{\abs[\big]{A \cap [-B,B[^d}}{(2B)^d}.
	\end{align*}
\end{definition}

In what follows, we will identify the module $R[x]_{\le n}$ of polynomials of
degree at most $n$ with $R^{n+1}$ by the standard basis $\{1, x, \ldots, x^n\}$.

\begin{definition}
	Let $E\subseteq \vZ^{n+1}$ be the set of degree $n$ Eisenstein polynomials
	over the integers.
	Let $E_p$ be the set of degree $n$ Eisenstein polynomials over $\vZ_p$.
\end{definition}

The reader should notice that we are computing the density of shifted (or
affine) Eisenstein polynomials of degree \emph{exactly} $n$ among polynomials
of degree \emph{at most} $n$.
Nevertheless it is easy to see that the density of shifted (and also affine, see
Remark~\ref{rem:eislowdeg}) Eisenstein polynomials of degree less or equal than
$n$ is the same.

\section{Shifted Eisenstein Polynomials}
\label{sec:shifted}

In this section, we determine the density of polynomials $f(x) \in \vZ[x]$ such
that $f(x+i)$ is Eisenstein for some shift $i \in \vZ$.
For this, let $\sigma$ be the linear map defined by
\begin{align*}
	\sigma\colon \vZ^{n+1} &\longrightarrow \vZ^{n+1} \\
	f(x) &\longmapsto f(x+1).
\end{align*}
It is easy to see that $\sigma$ has determinant one.
Similarly, we get a determinant one map over $\vZ_p$ for any $p$, which we will
also denote by $\sigma$.

\begin{definition}
	Let $\overline E\subseteq \vZ^{n+1}$ be the set of degree $n$ polynomials
	which are Eisenstein after applying some shift $i\in \vZ$:
	\begin{equation*}
		\overline E =
			\{f(x)\in\vZ^{n+1} \,:\, f(x+i)\in E \text{ for some } i \in \vZ \}.
	\end{equation*}
	We call these polynomials \emph{shifted Eisenstein}.
\end{definition}

In order to compute the density of $\overline{E}$, it we need to consider each
prime $p$ separately.
We do this by working over the $p$-adic integers.
\begin{definition}
	Let $\overline E_p \subseteq \vZ_p^{n+1}$ be the set of degree $n$ polynomials
	of $\vZ_p[x]$ which are Eisenstein after applying some shift $i\in \vZ_p$:
	\begin{equation*}
		\overline E_p =
			\{f(x)\in\vZ_p^{n+1} \,:\, f(x+i)\in E_p \text{ for some } i \in \vZ_p \}.
	\end{equation*}
	We also call these polynomials shifted Eisenstein, since it will always be
	clear from the context to which set we are referring.
\end{definition}

Notice that $\overline{E}_p\cap \vZ^{n+1}$ are exactly the polynomials of
$\vZ[x]$ of degree $n$ which are shifted $p$-Eisenstein.

\begin{lemma}\label{thm:disjoint_shifted}
	If $f(x) \in \vZ_p^{n+1}$ is shifted Eisenstein, then it is so with respect to
	exactly one rational integer shift $i \in \{0, \ldots, p-1\}$.
	In other words,
	\begin{equation*}
		\overline E_p = \bigsqcup_{i = 0}^{p-1} \sigma^{-i} E_p.
	\end{equation*}
\end{lemma}
\begin{proof}
	We clearly have
	\begin{equation*}
		\bigcup_{i = 0}^{p-1} \sigma^{-i} E_p \subseteq \overline E_p.
	\end{equation*}
	The other inclusion is easy but not completely trivial.

	Let $f(x) = \sum^n_{i=0} \alpha_i x^i \in E_p$ and $k \in \vZ_p$.
	We first show that $f(x + kp)$ is also Eisenstein:
	Clearly $f(x)=f(x+kp)$ in $\vZ_p/p\vZ_p$, so the only condition which one has
	to check is that the coefficient of the term of degree zero of $f(x+kp)$ is
	not in $p^2\vZ_p$.
	This coefficient is in fact
	$f(kp) = \alpha_0 + \alpha_1 k p + \sum^{n}_{i=2} \alpha_i k^i p^i$.
	Modulo $p^2\vZ_p$ we have that
	\begin{itemize}
		\item $\alpha_i k^i p^i$ is congruent to zero for $i\geq 2$,
		\item $\alpha_1 k p$ is congruent to zero since $\alpha_1$ is in $p\vZ_p$,
		\item $\alpha_0$ is not congruent to zero since the polynomial $f(x)$ is
			Eisenstein,
	\end{itemize}
	from which it follows that the polynomial $f(x+kp)$ is Eisenstein in
	$\vZ_p[x]$.

	Let now $f(x) \in \overline{E}_p$, then $f(x+u)$ is Eisenstein for some
	$u \in \vZ_p$.
	The inclusion
	\begin{equation*}
		\overline E_p \subseteq \bigcup_{i = 0}^{p-1} \sigma^{-i} E_p
	\end{equation*}
	will follow if we show that we can select $u$ in $\{0, \ldots, p-1\}$.
	Write $u = kp + i$ with $i \in \{0, \ldots, p-1\}$ and $k \in \vZ_p$.
	Using what we proved above, we see that $f(x + u - kp) = f(x + i)$ is
	Eisenstein, and the inclusion follows.

	We now show that the union is disjoint, i.e. $\sigma^{-i} E_p\cap \sigma^{-j}
	E_p=\emptyset$ for any $i,j\in \{0,\dots,p-1\}$ and $i\neq j$.
	Without loss of generality, we can assume $i>j$. Then
	\begin{equation*}
		\sigma^{-i} E_p \cap \sigma^{-j} E_p = \emptyset
		\:\Longleftrightarrow\: E_p \cap \sigma^{i-j} E_p = \emptyset.
	\end{equation*}
	Let $t:=j-i$ and $\sum^{n}_{i=0} \alpha_k x^k=f(x)\in E_p$, then the
	coefficient of the degree zero term of $f(x+t)$ is $f(t)=\alpha_n t^n+
	\sum^{n-1}_{k=0} \alpha_k t^k$.
	Now, the reduction of $\alpha_k$ modulo $p$ is zero for any $k<n-1$ and
	$\alpha_n$ and $t$ are invertible modulo $p$, so $f(x+t)$ is not Eisenstein.
\end{proof}

Let $\mu_p$ be the $p$-adic measure on $\vZ_p^{n+1}$ and $\mu_{\infty}$ the
Lebesgue measure on $\vR^{n+1}$. (For basics on the $p$-adic measure, we refer
to~\cite{bib:robert2013course}.)
\begin{lemma} In the above notation we have
	\[\mu_p(\overline E_p)=\frac{(p-1)^2}{p^{n+1}}.\]
\end{lemma}
\begin{proof}
	Since $\sigma^{-1}$ has determinant one, it does not change the $p$-adic
	volumes.
	Therefore, by Lemma~\ref{thm:disjoint_shifted}, one has $\mu_p(\overline
	E_p)=p\cdot \mu_p(E_p)$.
	It is easy to compute the measure $\mu_p(E_p)$ by writing $E_p = (p\vZ_p
	\setminus p^2\vZ_p) \times (p\vZ_p)^{n-1} \times (\vZ_p \setminus p\vZ_p)$.
\end{proof}

In order to obtain the density $\rho(\overline E)$ from the local data
$\{\mu_p(\overline E_p)\}_p$, we will use the following
lemma~\cite[Lemma~20]{bib:poonenAnn}.

\begin{lemma}\label{thm:bjornstoll}
	Suppose $U_\infty\subseteq \vR^d$ is such that $\vR^+\cdot U_\infty=U_\infty$,
	$\mu_\infty(\partial U_\infty)=0$.
	Let $U_\infty^1=U_\infty\cap [-1,1]^d$ and $s_\infty=\mu_\infty(U_\infty^1)$.
	Let $U_p\subseteq \vZ_p^d$, $\mu_p(\partial U_p)=0$ and $s_p=\mu_p(U_p)$ for
	each prime $p$.
	Let $M_\vQ$ be the set of places of $\vQ$.
	Moreover, suppose that
	\begin{equation}\label{eq:conditiondens}
		\lim_{M \to \infty} \rho(\{a \in \vZ^d \,:\, a\in U_p
			\text{ for some finite prime $p$ greater than $M$}\}) = 0.
	\end{equation}
	Let $P\colon \vZ^d \longrightarrow 2^{M_\vQ}$ be defined as
	$P(a) = \{v\in M_\vQ \,:\, a \in U_v\}$.
	Then we have:
	\begin{enumerate}
		\item $\sum_v s_v$ converges.
		\item For any $T \subseteq 2^{M_\vQ}$, $\nu(T):=\rho(P^{-1}(T))$ exists and
			defines a measure on $2^{M_\vQ}$, which is concentrated at the finite
			subsets of $M_\vQ$.
		\item Let $S$ be a finite subset of $M_\vQ$, then
			\begin{equation*}
				\nu(\{S\}) = \prod_{v \in S} s_v \prod_{v \notin S} (1-s_v).
			\end{equation*}
	\end{enumerate}
\end{lemma}
\begin{proof}
	For the proof, see~\cite[Lemma~20]{bib:poonenAnn}.
\end{proof}

After showing that condition (\ref{eq:conditiondens}) applies, we can use
Lemma~\ref{thm:bjornstoll} to determine the density of shifted Eisenstein
polynomials over the integers.

\begin{theorem}\label{thm:shifted3}
Let $n \ge 3$.
The density of shifted Eisenstein polynomials of degree $n$ is
\begin{equation}\label{eq:shifted3}
	\rho(\overline E) = 1 - \smashoperator{\prod_{p \text{ prime}}} \;
		\left(1-\frac{(p-1)^2}{p^{n+1}}\right).
\end{equation}
\end{theorem}
\begin{proof}
Set $U_p=\overline{E}_p$ for all $p$ and $U_\infty=\emptyset$.
The conditions $\mu_p(\partial U_p) = 0$ hold since $U_p$ is both closed and
open.
Notice that in the notation of Lemma~\ref{thm:bjornstoll} we have that
$P^{-1}(\{\emptyset\})$ equals the complement of $E$.
Therefore, if condition (\ref{eq:conditiondens}) is verified, we get the claim:
\begin{equation*}
	\rho(\overline E) = 1-\smashoperator{\prod_{p \text{ prime}}}
		\left(1-s_p\right)
	= 1-\smashoperator{\prod_{p \text{ prime}}} \;
		\left(1-\frac{(p-1)^2}{p^{d+1}}\right).
\end{equation*}

Let us now show that the condition indeed holds:
\begin{gather}
	\lim_{M \to \infty}\overline{\rho}
		(\{a\in \vZ^{n+1} \,:\, a\in \overline{E}_p
		\text{ for some finite prime $p$ greater than $M$}\}) \nonumber \\
	= \lim_{M \to \infty}\limsup_{B \to \infty}
		\frac{\abs[\big]{\bigcup_{p>M}\overline{E}_p\cap [-B,B[^{n+1}}}{(2B)^{n+1}}.
		\label{eq:fundamentallimit}
\end{gather}
We have $\overline E_p \cap [-B,B[^{n+1}\,=\emptyset$ for $p>CB^2$, where $C$ is
a constant depending only on the degree $n$.
One can see that using the following argument:
Let $f(x)$ be a polynomial in $[-B,B[^{n+1}$ for which $f(x+i)$ is Eisenstein,
then~\cite[Lemma~1]{bib:heyman2014shifted}
\begin{equation*}
	p^{n-1} \mid \disc(f(x+i)) = \disc(f(x)) \ne 0.
\end{equation*}
Now, the discriminant of $f(x)$ is a polynomial of degree $2n-2$ in the
coefficients, whence
\begin{equation*}
	p^{n-1}\leq\disc(f(x))\leq D B^{2n-2}
\end{equation*}
for some constant $D$ depending only on $n$.
Therefore, for $C=D^{1/(n-1)}$, we have $p\leq C B^{2}$.
Thus, we have just shown that for fixed $B$, the union in
(\ref{eq:fundamentallimit}) is finite, and we can bound it by
\begin{gather}
	\lim_{M \to \infty}\limsup_{B \to \infty}
		\frac{\abs[\big]{\bigcup_{CB^2>p>M}\overline{E}_p \cap [-B,B[^{n+1}}}
		{(2B)^{n+1}} \nonumber \\
	\leq \lim_{M \to \infty}\limsup_{B \to \infty}
		\smashoperator[r]{\sum_{CB^2>p>M}}
		\frac{\abs[\big]{\overline{E}_p\cap [-B,B[^{n+1}}}{(2B)^{n+1}}.
		\label{eq:equationsum}
\end{gather}
Given the order of the limits, we can fix the following setting:
$M>n$ and $B>M$.
Now let us bound $\abs[\big]{\overline{E}_p\cap [-B,B[^{n+1}}$ in the following
two cases:
\begin{enumerate}
\item $2B<p$:
In this case, we can consider $[-B,B[^{n+1}$ as a subset of
$\smash{\vF_p^{n+1}}$ without losing any information.
The reader should notice that modulo $p$, the elements of $\overline{E}_p$ have
a multiple root of order $n$ at some $i\in \vF_p$.
Now, the key observation is the following:
The reduction modulo $p$ of the polynomials in $[-B,B[^{n+1}\;\cap\;
\overline{E}_p$ is contained in the set
\begin{equation*}
	S_p := \{a(x-i)^n \,:\,
		a \in [-B,B[\,\setminus \{0\} \text{ and } {-nai} \in [-B,B[\}.
\end{equation*}
This represents the condition that the degree $n$ and $n-1$ coefficients live in
$[-B,B[$:
\begin{equation*}
		[-B,B[^{n+1} \,\cap\, \overline E_p\subseteq S_p.
\end{equation*}
Observe now that $\abs{S_p}=(2B-1)2B\leq (2B)^2$, since $n$ and $a$ are
invertible modulo $p$ (recall $p>M>n$).
We conclude that
\begin{equation*}
	\abs[\big]{[-B,B[^{n+1} \,\cap\, \overline E_p}\leq \abs{S_p}\leq (2B)^2.
\end{equation*}
Notice that this bound is uniform in $p$.

\item $2B\geq p$: In this case, the bound is more natural.
Consider the projection map
\begin{align*}
	\pi\colon \vZ^{n+1} &\longrightarrow \vF_p^{n+1} \\
	\intertext{and the shift map modulo $p$}
	\sigma^{-1}\colon \vF_p^{n+1} &\longrightarrow \vF_p^{n+1} \\
	f(x) &\longmapsto f(x-1).
\end{align*}

Consider the sets of polynomials $L_p := \{ax^n \,:\, a\in\vF_p^*\}$ and
\begin{equation}\label{eq:Lp_union}
	\overline{L}_p = \bigcup^{p-1}_{i=0} \sigma^{-i}L_p.
\end{equation}
We have $\abs{\overline L_p}\leq p^2$.

Notice that
\begin{equation}\label{eq:observation}
\pi([-B,B[^{n+1} \,\cap\, \overline E_p)\subseteq \overline{L}_p.
\end{equation}
At this step, we observe that the projection is at most $\ceil{2B/p}^{n+1}$ to
one, therefore we can bound $\abs[\big]{[-B,B[^{n+1} \,\cap\, \overline E_p}$
using the projection map and condition (\ref{eq:observation}):
\begin{equation*}
	\abs[\big]{[-B,B[^{n+1} \,\cap\, \overline E_p}
	\leq \abs{\overline{L}_p} \cdot \ceil{2B/p}^{n+1}
	\leq p^2 \left(\frac{2B}{p}+1\right)^{n+1}
	\leq p^2 \left(\frac{4B}{p}\right)^{n+1},
\end{equation*}
where the last inequality follows from $2B\geq p$.
At the end of the day, the bound we have is of the form
\begin{equation*}
	\abs[\big]{[-B,B[^{n+1} \,\cap\, \overline E_p}\leq 4^{n+1}
	\frac{B^{n+1}}{p^{n-1}}.
\end{equation*}
\end{enumerate}

Let us now come back to the sum in (\ref{eq:equationsum}), which we can split
according to the two cases above:
\begin{equation}\label{eq:estimate_sum}
	\smashoperator[r]{\sum_{CB^2>p>M}}
		\frac{\abs[\big]{\overline{E}_p \,\cap\, [-B,B[^{n+1}}}{(2B)^{n+1}}
	\leq \sum_{CB^2>p>2B}\frac{(2B)^2}{(2B)^{n+1}}
		+ \sum_{2B \ge p>M}\frac{2^{n+1}}{p^{n-1}}.
\end{equation}
Using the limit in $B$, the first sum goes to zero by the prime number theorem
since $n\geq 3$.
As $B$ goes to infinity, the other sum becomes a converging series (again
$n\geq 3$) starting at the index $M$.
Letting $M$ go to infinity, this too goes to zero.
Hence we have shown that condition~(\ref{eq:conditiondens}) holds, and the
theorem follows.
\end{proof}

In degree $2$, the above proof does not work:
Indeed, it is easily seen that $\sum_p s_p$ diverges for $n=2$, so by the first
claim of Lemma~\ref{thm:bjornstoll}, the proof we gave in degree greater or
equal than $3$ is doomed to fail in degree $2$.
However, we have a much simpler application of the lemma which shows that the
density of shifted Eisenstein polynomials of degree $2$ is indeed one, as
Theorem~\ref{thm:shifted3} suggests.

\begin{proposition}\label{thm:shifted2}
	The density of shifted Eisenstein polynomials of degree $n = 2$ is one.
\end{proposition}
\begin{proof}
	Let again $U_\infty = \emptyset$.
	We now apply Lemma~\ref{thm:bjornstoll} to a truncated sequence of sets.
	For this, let $M$ be a positive integer and
	\begin{equation*}
		U_p = \begin{cases}
			\overline{E}_p &\text{ if } p \le M \\
			\emptyset &\text{ if } p > M.
		\end{cases}
	\end{equation*}
	This truncated sequence now automatically satisfies
	condition~(\ref{eq:conditiondens}), and we get the density
	\begin{equation*}
		\underline\rho(\overline{E})
		\ge \rho\Big(\smashoperator[r]{\bigcup_{p \le M}}
			\overline{E}_p \cap \vZ^{3}\Big)
		= 1 - \smashoperator{\prod_{p \le M}}\left(1-\frac{(p-1)^2}{p^{3}}\right).
	\end{equation*}
	Letting $M$ tend to infinity gives $\rho(\overline{E}) = 1$, as the product
	diverges to zero.
\end{proof}

\begin{remark}
	Even though the density of shifted Eisenstein polynomials of degree $2$ is
	one, not all irreducible polynomials are Eisenstein for some shift (or even
	affine transformation):
	Take for example the polynomial $f(x) = x^2 + 8x - 16$, which is irreducible
	over $\vZ$.
	Its discriminant is $2^7$, so it could only be shifted Eisenstein with respect
	to $2$.
	But neither $f(x)$ nor $f(x+1) = x^2 + 10x - 7$ is $2$-Eisenstein.
\end{remark}

\section{Affine Eisenstein Polynomials}
\label{sec:affine}

In~\cite[Section~7]{bib:heyman2014shifted}, the question was also raised about
the density of polynomials that become Eisenstein after an arbitrary affine
transformation, instead of only considering shifts.
We can address this question as well, using the same methods as in
Section~\ref{sec:shifted}.

\begin{definition}\label{def:affine_transformation}
	For $f(x) \in R^{n+1}$ and
	$A = \begin{psmallmatrix} a & b \\ c & d \end{psmallmatrix}\in R^{2\times 2}$,
	we define the \emph{affine transformation of $f$ by $A$} as
	\begin{equation*}
		f * A := (cx + d)^n f\bigg(\frac{ax + b}{cx + d}\bigg).
	\end{equation*}
	It is easy to see that, when restricted to $\GL_2(R)$, this is a right
	group action.
\end{definition}

Like in Section~\ref{sec:shifted}, we consider the set of polynomials with
integer coefficients that become Eisenstein after some affine transformation.

\begin{definition}
	Let $\widetilde{E} \subseteq \vZ^{n+1}$ be the set of degree $n$ polynomials
	which become Eisenstein of degree $n$ after some affine transformation
	$A\in \vZ^{2 \times 2}$:
	\begin{equation*}
		\widetilde{E} = \{f(x)\in\vZ^{n+1} \,:\, f * A \in E \text{ for some }
			A \in \vZ^{2 \times 2} \}.
	\end{equation*}
	We call these polynomials \emph{affine Eisenstein}.
\end{definition}

It is easy to see that if both $f$ and $f * A$ have degree $n$ and $f * A$ is
irreducible, then so is $f$.
Hence, an affine Eisenstein polynomial is irreducible.

\begin{remark}\label{rem:eislowdeg}
	The reader should notice that also in this case, we only consider affine
	Eisenstein polynomials of degree \emph{exactly} $n$.
	Nevertheless an observation is required: It could happen that a degree $n$
	polynomial becomes Eisenstein of \emph{lower} degree after some affine
	transformation.
	Fortunately, it can be seen that a polynomial for which this happens is never
	irreducible.
	Likewise, a polynomial of degree less than $n$ cannot become Eisenstein of
	degree $n$ after an affine transformation, since any transformation that
	increases the degree introduces factors $cx + d$.
\end{remark}

We again consider each prime separately by working over the $p$-adic integers.

\begin{definition}
	Let $\widetilde{E}_p \subseteq \vZ_p^{n+1}$ be the set of degree $n$
	polynomials of $\vZ_p[x]$ which become Eisenstein of degree $n$ after some
	affine transformation $A \in \vZ_p^{2 \times 2}$:
	\begin{equation*}
		\widetilde{E}_p = \{f(x)\in\vZ_p^{n+1} \,:\, f * A \in E_p \text{ for some }
			A \in \vZ_p^{2 \times 2} \}.
	\end{equation*}
	We also call these polynomials affine Eisenstein, since it will always be
	clear from the context to which set we are referring.
\end{definition}

In what follows we compute the measure $\mu_p(\widetilde E_p)$.
For this, we need to write $\widetilde{E}_p$ as a disjoint union of transformed
copies of $E_p$ as in Lemma~\ref{thm:disjoint_shifted}.
The following lemma is essential for this.

\begin{lemma}\label{thm:affine_equivalence}
	Assume $f(x) \in \vZ_p^{n+1}$ is Eisenstein of degree $n$, and let
	$A = \begin{psmallmatrix} a & b \\ c & d \end{psmallmatrix}
	\in \vZ_p^{2 \times 2}$.
	Then, $f * A$ is Eisenstein of degree $n$ if and only if $p \mid b$, $p \nmid
	a$, $p \nmid d$.
\end{lemma}
\begin{proof}
	If we write $f(x) = \sum_{i=0}^{n} \alpha_i x^i$ and $f * A = \sum_{l=0}^{n}
	\beta_l x^l$, then a simple calculation gives
	\begin{equation}\label{eq:affine_coefficients}
		\beta_l = \sum_{j=0}^{l} \sum_{s=l}^{n} \binom{n + j - s}{j}
		\binom{s-j}{l-j} \alpha_{s-j} d^{n-s} b^{s-l} a^{l-j} c^j.
	\end{equation}

	Assume now that $f * A$ is Eisenstein, so $p \mid \beta_l$ for $0 \le l \le
	n-1$, $p^2 \nmid \beta_0$, $p \nmid \beta_n$.
	Consider first $\beta_0$.
	Reducing modulo $p$ and using that $p \mid \alpha_i$ for $i < n$, we see that
	\begin{equation*}
		\beta_0 \equiv \alpha_n b^n \pmod{p}.
	\end{equation*}
	Since $p \nmid \alpha_n$, we get that $p \mid b$.
	Knowing this, we reduce $\beta_0$ modulo $p^2$ and get
	\begin{equation*}
		\beta_0 \equiv \alpha_0 d^n + \alpha_1 d^{n-1} b
		\equiv \alpha_0 d^n \pmod{p^2},
	\end{equation*}
	since $p^2 \mid \alpha_1 b$.
	From this, we see that $p^2 \nmid \beta_0$ if and only if $p \nmid d$.

	Finally, we reduce $\beta_n$ modulo $p$ and get
	\begin{equation*}
		\beta_n \equiv \alpha_n a^n \pmod{p},
	\end{equation*}
	from which we conclude that $p \nmid a$.

	Vice versa, if we assume that $p \mid b$, $p \nmid a$, $p \nmid d$, the same
	computations as above show that $p \nmid \beta_n$, $p \mid \beta_0$, $p^2
	\nmid \beta_0$, and we easily see from (\ref{eq:affine_coefficients}) that $p
	\mid \beta_l$ for $0 < l < n$.
	Hence, $f * A$ is Eisenstein.
\end{proof}

We denote by $S = \{ \begin{psmallmatrix} a & b \\ c & d \end{psmallmatrix}
\in \vZ_p^{2 \times 2} \,:\, p \mid b, p \nmid a, p \nmid d \}$ the set of
matrices from Lemma~\ref{thm:affine_equivalence}.
This is a subgroup of $\GL_2(\vZ_p)$.
We can obtain the disjoint union decomposition of $\widetilde{E}_p$ by
considering the left cosets of $S$, but first, we need to deal with the
noninvertible matrices.
It turns out that they don't matter.

\begin{lemma}\label{thm:noninvertible}
	Let $f(x) \in \vZ_p^{n+1}$.
	If $A = \begin{psmallmatrix} a & b \\ c & d \end{psmallmatrix}
	\in \vZ_p^{2 \times 2}$ is \emph{not} invertible, then $f * A$ is not
	Eisenstein of degree $n$.
\end{lemma}
\begin{proof}
	Assume for contradiction that $f * A$ is Eisenstein of degree $n$.
	We write again $f(x) = \sum_{i=0}^{n} \alpha_i x^i$ and $f * A =
	\sum_{l=0}^{n} \beta_l x^l$.
	We reduce modulo $p$:
	\begin{align*}
		\bar{A} = \begin{pmatrix}
			\bar{a} & \bar{b} \\ \bar{c} & \bar{d}
		\end{pmatrix} \in \vF_p^{2 \times 2}.
	\end{align*}
	Since $\det \bar{A} = 0$, there are two cases:
	Either $\bar{c} = \bar{d} = 0$, or there is a $\lambda \in \vF_p$ such that
	$\bar{a} = \lambda \bar{c}$ and $\bar{b} = \lambda \bar{d}$.

	We consider the second case.
	Since $f * A$ is Eisenstein, we see that $\bar{f} * \bar{A} = \bar \beta_n x^n
	\in \vF_p[x]$ with $\bar{\beta}_n \ne 0$.
	On the other hand,
	\begin{equation*}
		\bar{f} * \bar{A} = (\bar{c} x + \bar{d})^n \bar{f}
		\bigg(
			\frac{\lambda \bar{c} x + \lambda \bar{d}}{\bar{c} x + \bar{d}}
		\bigg)
		= (\bar{c} x + \bar{d})^n \bar{f}(\lambda).
	\end{equation*}
	From this, we see that $\bar{f}(\lambda) \ne 0$, $\bar{c} \ne 0$ and
	$\bar{d} = 0$.
	This means that $p \mid d$ and $p \mid b$, from which it follows by
	(\ref{eq:affine_coefficients}) that $p^2 \mid \beta_0$.
	This contradicts the assumption that $f * A$ is Eisenstein.

	The case $\bar{c} = \bar{d} = 0$ is similar.
\end{proof}

Hence, we only need to consider the action of $\GL_2(R)$ on $R^{n+1}$.
According to Lemma~\ref{thm:affine_equivalence}, the action of elements of $S$
does not change whether a polynomial is Eisenstein.
Therefore, to see if a polynomial $f(x) \in \vZ_p^{n+1}$ is affine Eisenstein,
it is enough to check one representative of each left coset of $S \subset
\GL_2(\vZ_p)$.
We can list these cosets explicitly.

\begin{lemma}\label{thm:equivalence_classes}
	The subgroup $S \subset \GL_2(\vZ_p)$ has $p + 1$ left cosets, which are the
	following:
	\begin{itemize}
		\item $\begin{pmatrix} 1 & i \\ 0 & 1 \end{pmatrix}S$ for
			$i \in \{0, \ldots, p-1\}$ (corresponding to shifts), and
		\item $\begin{pmatrix} 0 & 1 \\ 1 & 0 \end{pmatrix}S$ (corresponding to
			the reciprocal).
	\end{itemize}
\end{lemma}
\begin{proof}
	It is easy to see that these $p+1$ left cosets are distinct.
	We need to show that every
	$A = \begin{psmallmatrix} s & t \\ u & v \end{psmallmatrix}\in \GL_2(\vZ_p)$
	lies in one of them.

	Consider first the case $p \mid v$.
	Then,
	\begin{equation*}
		\begin{pmatrix}
			s & t \\ u & v
		\end{pmatrix}
		=
		\begin{pmatrix}
			0 & 1 \\ 1 & 0
		\end{pmatrix}
		\begin{pmatrix}
			u & v \\ s & t
		\end{pmatrix},
	\end{equation*}
	with $\begin{psmallmatrix} u & v \\ s & t \end{psmallmatrix} \in S$.

	If instead $p \nmid v$, let $i \equiv t/v \pmod{p}$,
	$i \in \{0, \ldots, p-1\}$.
	Then,
	\begin{equation*}
		\begin{pmatrix}
			s & t \\ u & v
		\end{pmatrix}
		=
		\begin{pmatrix}
			1 & i \\ 0 & 1
		\end{pmatrix}
		\begin{pmatrix}
			s - iu & t - iv \\ u & v
		\end{pmatrix},
	\end{equation*}
	with $p \mid t - iv$ by choice of $i$, and $p \nmid s - iu$ since the matrix
	has to be invertible.
\end{proof}

Together, Lemmata~\ref{thm:affine_equivalence} and~\ref{thm:equivalence_classes}
say that $f(x)$ is affine Eisenstein with respect to some $A$ if and only if it
is shifted Eisenstein with respect to some $i \in \{0, \ldots, p-1\}$, or if its
reciprocal $x^n f(1/x)$ is Eisenstein; and these possibilities are exclusive.
In other words,
\begin{equation*}
	\widetilde E_p = \operatorname{recip}(E_p) \sqcup \bigsqcup_{i = 0}^{p-1}
	\sigma^{-i} E_p.
\end{equation*}
Since shifting and taking the reciprocal are linear maps with determinant $\pm
1$, they preserve the $p$-adic measure, and we see that
\begin{equation*}
	\mu_p(\widetilde E_p) = (p+1) \mu_p(E_p) = \frac{(p+1)(p-1)^2}{p^{n+2}}.
\end{equation*}
With this, we can now show the analogue of Theorem~\ref{thm:shifted3} for affine
transformations.

\begin{theorem}\label{thm:affine3}
	Let $n \ge 3$.
	The density of affine Eisenstein polynomials of degree $n$ is
	\begin{equation*}
		\rho(\widetilde E) = 1 - \smashoperator{\prod_{p \text{ prime}}} \;
		\left(1 - \frac{(p+1)(p-1)^2}{p^{n+2}}\right).
	\end{equation*}
\end{theorem}
\begin{proof}
	The proof is mostly the same as for Theorem~\ref{thm:shifted3}.
	For the verification of condition (\ref{eq:conditiondens}), note that the case
	$2B < p$ is unchanged from the proof of Theorem~\ref{thm:shifted3}, since the
	reciprocal polynomial cannot be $p$-Eisenstein for $p > B$.
	For the case $2B \ge p$, we simply get an additional term in the union
	(\ref{eq:Lp_union}), and so the estimate changes to $\abs{\overline L_p} \le
	p(p+1)$.
	However, this doesn't affect the convergence of the sum in
	(\ref{eq:estimate_sum}).
\end{proof}

\begin{remark}
	Clearly, the density of affine Eisenstein polynomials of degree $n = 2$ is
	one.
	After all, we are considering a superset of the shifted Eisenstein polynomials
	of Proposition~\ref{thm:shifted2}.
\end{remark}

\section{Monte Carlo Simulations}
\label{sec:simulations}

As in~\cite[Section~6]{bib:heyman2014shifted}, we ran some Monte Carlo
simulations to verify how near our results are to the actual probability of
finding a shifted (or affine) Eisenstein polynomial among all the polynomials of
a given height.
For degrees $n=3$ and $4$, we tested $20\,000$ random polynomials of height at
most $1\,000\,000$.
The results are shown in Tables~\ref{tab:simulation3} and~\ref{tab:simulation4}.
The first column contains the number of polynomials which were actually found by
the Monte Carlo experiment, while the second column contains the expected number
given by~\cite[Theorem~2]{bib:shparlinskiEisen} and Theorems~\ref{thm:shifted3}
and~\ref{thm:affine3}.
All the experiments seem to agree with our theoretical results.

The simulations were done using the Sage computer algebra
system~\cite{bib:sage}, and the code is available upon request.

\begin{table}[htp]
	\centering
	\caption{Simulations for degree $n = 3$.}
	\label{tab:simulation3}
	\begin{tabular}{|l|r|r|}
		\hline
		{}                   &     found &   expected \\
		\hline
		irreducible          & $20\,000$ &  $20\,000$ \\
		Eisenstein           &    $1112$ &     $1112$ \\
		shifted Eisenstein   &    $3416$ &     $3353$ \\
		affine Eisenstein    &    $4360$ &     $4328$ \\
		\hline
	\end{tabular}
\end{table}

\begin{table}[htp]
	\centering
	\caption{Simulations for degree $n = 4$.}
	\label{tab:simulation4}
	\begin{tabular}{|l|r|r|}
		\hline
		{}                   &     found &   expected \\
		\hline
		irreducible          & $20\,000$ &  $20\,000$ \\
		Eisenstein           &     $432$ &      $449$ \\
		shifted Eisenstein   &    $1096$ &     $1112$ \\
		affine Eisenstein    &    $1570$ &     $1547$ \\
		\hline
	\end{tabular}
\end{table}

\section*{Acknowledgements}

The authors were supported in part by Swiss National Science Foundation grant number 149716 and \emph{Armasuisse}.

\bibliographystyle{plainnat}
\bibliography{biblio}{}

\end{document}